\documentclass[11pt]{article}
\usepackage{subfigure}
\usepackage{authblk}
\usepackage{hyperref}

\usepackage{amssymb,amsthm,amsmath,mathrsfs,fullpage,multirow}
\usepackage{tikz}
\usetikzlibrary{matrix}

\newcommand{\Z}{\mathbb{Z}}

\renewcommand\S{\mathcal S}

\newtheorem{theorem}{Theorem}[section]
\newtheorem{lemma}[theorem]{Lemma}
\newtheorem{proposition}[theorem]{Proposition}

\newtheorem{conjecture}[theorem]{Conjecture}

  \newtheorem{question}[theorem]{Question}

\usepackage{enumerate}

\title{Pattern avoidance and the fundamental bijection}
\author[1]{Kassie Archer}
\author[1]{Robert P. Laudone}
\affil[1]{{\small Department of Mathematics, United States Naval Academy, Annapolis, MD 21402}}
\affil[ ]{{\small Email: karcher@usna.edu, laudone@usna.edu}}
\date{}

\begin{document}

\maketitle

\begin{abstract}
The fundamental bijection is a bijection $\theta:\mathcal{S}_n\to\mathcal{S}_n$ in which one uses the standard cycle form of one permutation to obtain another permutation in one-line form. In this paper, we enumerate the set of permutations $\pi\in\mathcal{S}_n$ that avoids a pattern $\sigma\in\mathcal{S}_3$, whose image $\theta(\pi)$ also avoids $\sigma$. We additionally consider what happens under repeated iterations of $\theta$; in particular, we enumerate permutations $\pi\in\mathcal{S}_n$ that have the property that $\pi$ and its first $k$ iterations under $\theta$ all avoid a pattern $\sigma$. Finally, we consider permutations with the property that $\pi=\theta^2(\pi)$ that avoid a given pattern $\sigma$, and end the paper with some directions for future study.
\end{abstract}

\section{Introduction}


The problem of understanding how cycle type and other algebraic properties of a permutation connect to pattern avoidance in one-line notation has proven challenging. For example, it is still an open question to determine the number of cyclic permutations that avoid a single pattern of length three in their one-line notation. Nevertheless, some progress on this general problem has been made. This progress includes enumeration of pattern-avoiding permutations composed of cycles of a restricted size \cite{AG21, AL, BD19, DRS07, GM02}, cyclic permutations avoiding sets of patterns \cite{A22, AL20, BC19, H19}, and cyclic permutations avoiding a pattern in both its one-line and cycle form \cite{ABBGJ23}. 

The work done in \cite{ABBGJ23} involves avoidance in a cyclic permutation and cycle form.  In this paper, we study all pattern-avoiding permutations and their image under the {\em fundamental bijection}. This bijection (described in \cite[Page 30]{S11} and in Section~\ref{sec:prelim} of this paper) involves writing a permutation in its standard cycle form, and reading off a new permutation in its one-line notation. In a sense, by asking a permutation and its image under the fundamental bijection to both avoid a given pattern, we are requiring a permutation to avoid that pattern in both its one-line form and its cycle form.

Pattern avoidance in the image of the fundamental bijection has also recently made an appearance in \cite{BT23} as a way of  characterizing so-called shallow permutations.
In \cite{DG77}, Diaconis and Graham consider metrics for a permutation $\pi$: the number of inversions $i(\pi)$, the total displacement $d(\pi)$, and the reflection length $t(\pi)$. They found that for any permutation $\pi,$ $i(\pi)+t(\pi)\leq d(\pi)\leq 2i(\pi),$ and they show that permutations that satisfy  that the upper bound is an equality are exactly those permutations that avoid 321. Ones where equality holds for the lower bound are called \emph{shallow permutations}, and have recently been shown to be equivalent to so-called unlinked permutations, whose cycle diagram corresponds to the knot diagram of an unlink \cite{W22, CM24}.
In \cite{BT23}, Berman and Tenner prove that the shallow permutations can also be characterized in terms of pattern avoidance; in particular, shallow permutations are those permutations $\pi$ so that the image of $\pi$ under the fundamental bijection avoids a set of vincular patterns. They use this characterization to enumerate shallow cycles and involutions.

In this paper, we consider how avoidance in one-line notation interacts with the fundamental bijection. In particular, we enumerate the set of permutations avoiding a given pattern of size 3 whose image under the fundamental bijection also avoids this pattern (Section \ref{sec: SinglePattern}). The nontrivial results are summarized in Table \ref{tab:onepat}. We then ask the same question for higher interations of the fundamental bijection: how many permutations avoid a given pattern and continue to avoid that pattern under $k$ compositions of the fundamental bijection? This turns out to be a challenging question; we prove some results and conjecture the answer to others (Section \ref{sec: HigherIterations}). Next, we study the number of permutations avoiding a pattern of size three that are fixed by $k$ iterations of the fundamental bijection; meaning after $k$ iterations of the fundamental bijection we return to the original permutation (Section \ref{sec: fixedOrder}). We answer this question completely for permutations fixed after two iterations avoiding a single pattern. We then notice and conjecture some mysterious periodicity as $k$ increases. We conclude by providing a lower bound on the number of permutations fixed after two iterations, along with more open questions (Section \ref{sec: Conclusion}).

\section{Preliminaries and Notation}\label{sec:prelim}

Let $\S_n$ denote the symmetric group on $[n] = \{1,2,\dots,n\}$. There are many ways to represent a permutation $\pi \in \S_n$. One common way is one-line notation, $\pi = \pi_1\pi_2 \ldots \pi_n$ where $\pi_i = \pi(i)$. It is in this notation that one most commonly discusses avoidance and containment in a permutation. We say that a permutation $\pi$ \emph{contains} $\tau = \tau_1\tau_2\ldots \tau_k$ if there are indices $i_1 < i_2 < \cdots < i_k$ with $\pi_{i_r} < \pi_{i_s}$ if and only if $\tau_r < \tau_s$. If $\pi$ does not contain $\tau$, we say $\pi$ \emph{avoids} $\tau$. For example, the permutation $\pi = 41253$ contains the permutation $213$ since 415 is a subsequence of $\pi$ in the same order as $213$ and it avoids $321$ since there is no length 3 decreasing subsequence of $\pi$.

Another way to represent $\pi$ is using \emph{cycle notation}, expressing it as a product of disjoint cycles. In this paper, we will always express permutation in their \emph{standard cycle notation}, in which we write cycles with their largest element first, and order the cycles by that largest element. For example, the permutation $\pi = 41352987$ written in its one-line notation can also be expressed in standard cycle notation as $\pi = (3)(5,2,1,4)(8)(9,7)$. 

The \emph{fundamental bijection} $\theta:\S_n\to\S_n$ (described in \cite[Page 30]{S11}) utilizes both the one-line and standard cycle notation of permutations.
For a permutation $\pi \in \S_n$, we obtain $\theta(\pi)$ by writing $\pi$ in its standard cycle notation and reading a new permutation $\theta(\pi)$ in its one-line notation, i.e., by removing the parentheses from the standard cycle notation of $\pi$. For example, using the example in the previous paragraph, we can see that $\theta(41352987) = 35214897.$
For clarity of exposition, we will often write $\theta(\pi)$ as $\hat\pi$. 

We let $\mathcal{T}_n(\sigma)$ denote the set of permutations $\pi\in\S_n$ so that both $\pi$ and $\theta(\pi)$ avoid a pattern $\sigma$. For example, we have $41352987\in \mathcal{T}_n(4321)$ since both 41352987 and $\theta(41352987)=35214897$ avoid the pattern 4321. We let $t_n(\sigma)=|\mathcal{T}_n(\sigma)|$ and  we  denote the generating function for these values as $T_\sigma(x) = \sum_{n\geq 0} t_n(\sigma)x^n.$ 

Given $\pi \in \S_n$, we denote by $\pi^c$ the {\em complement} of $\pi$ obtained by replacing $i$ in the one line notation for $\pi$ with $n-i+1$ for each $1\leq i \leq n$. For example, if $\pi = 4132$ then $\pi^c = 1423$. We denote by $\pi^r$ the {\em reverse} of $\pi$ obtained by reading the one-line notation of $\pi$ in reverse (so that $\pi_i^r = \pi_{n+1-i})$. For example, if $\pi = 4132$ then $\pi^r = 2314$. We will sometimes combine these operations and for ease of notation will denote the reverse complement by $\pi^{rc}$, which we note is also equal to $\pi^{cr}$. 

Given two permutations $\pi \in \S_n$ and $\tau \in \S_m$, the {\em direct sum} of $\pi$ and $\tau$, denoted $\pi \oplus \tau \in \S_{n+m}$ is the permutation defined by
\[
(\pi \oplus \tau)(i) =
\begin{cases}
    \pi(i) &1 \leq i \leq n\\
    \tau(i-n) + n & n+1 \leq i \leq n+m
\end{cases}.
\]
For example, $3412\oplus321 = 3412765$ and $21\oplus 1\oplus 312\oplus1 = 2136457$. As shorthand, we will write a permutation obtained from a series of direct sums $\pi=\tau_1\oplus\tau_2\oplus\cdots\oplus\tau_k$ as $\pi=\bigoplus_{i=1}^k\tau_i.$

We say a permutation $\pi \in \S_n$ is {\em reducible} if $\pi = \sigma \oplus \tau$ for $\sigma \in \S_{n-k}$ and $\tau \in \S_k$ with $1 \leq k \leq n-1$. We will call a permutation $\pi$ {\em irreducible} (sometimes called indecomposable or connected) if it is not reducible. For example, $4123$ is irreducible while $132 = 1 \oplus 21$ is reducible. The following result about how $\theta$ interacts with direct sum decompositions will be useful in the coming sections.

\begin{lemma} \label{lem: thetaCommutes}
    Given permutations $\pi \in \S_n$ and $\tau \in \S_m$, $\theta(\pi \oplus \tau) = \theta(\pi) \oplus \theta(\tau)$.
\end{lemma}

\begin{proof}
Since all elements in $\{1,2,\ldots,n\}$ map to themselves and all elements $\{n+1,\ldots, n+m\}$ similarly map to themselves in the permutation $\pi\oplus\tau$, the cycles involving the elements in $\{1,2,\ldots,n\}$ will appear before those involving $\{n+1,\ldots, n+m\}$ in the standard cycle form of $\pi\oplus\tau,$ and thus in $\theta(\pi\oplus\tau)$. The result must follow.
\end{proof}

An immediate corollary of this result is that $\theta(\bigoplus_i \pi_i) = \bigoplus_i \theta(\pi_i)$; that is, $\theta$ commutes with direct sums. Another nice property of $\theta$ is that it preserves irreducibility, as stated in the proposition below.

\begin{proposition}
    Given $\pi \in \S_n$, $\pi$ is irreducible if and only if $\theta(\pi)$ is irreducible.
\end{proposition}

\begin{proof}
    Since $\theta$ is a bijection, it suffices to prove that $\pi$ is reducible if and only if $\theta(\pi)$ is. This follows immediately from Lemma \ref{lem: thetaCommutes} since $\pi$ is reducible if and only if $\pi = \sigma \oplus \tau$. By Lemma \ref{lem: thetaCommutes} and the fact that $\theta$ is a bijection, this happens if and only if $\theta(\pi) = \theta(\sigma) \oplus \theta(\tau)$.

    This means that $\theta$ is a bijection from reducible permutations to reducible permutations and therefore must also be a bijection from irreducible permutations to irreducible permutations.
\end{proof}

\section{Avoiding a single pattern} \label{sec: SinglePattern}

In this section, we enumerate the set of permutations $\pi\in\S_n$ so that both $\pi$ and $\theta(\pi)$ avoid a pattern $\sigma\in\S_3$.
A summary of these results can be found in Table~\ref{tab:onepat}.

\renewcommand{\arraystretch}{2.2}
 \begin{table}[ht]
            \begin{center}
\begin{tabular}{|c|c|c|}
    \hline
    $\sigma$ & $t_n(\sigma)$ & Theorem \\
    \hline
    \hline
   $123$ & 0
    &
    Theorem~\ref{thm:123}  \\[4pt]
    \hline
    132 & $\displaystyle\sum_{k=0}^{n+1} F_{n+1-k} F_k$  & Theorem~\ref{thm:132} \\[4pt]
    \hline
    213 & $2 F_{n+2} + n^2 - 6n +4$ & Theorem~\ref{thm:213} \\[4pt]
    \hline
    231 & $2^{n-1}$ & Theorem~\ref{thm:231} \\[4pt]
    \hline
    312 & $2^{n-1}$ & Theorem~\ref{thm:312} \\[4pt]
    \hline
    321 & g.f. $T_{321}(x) = \dfrac{2x^2}{2x(1-x)-1+\sqrt{1-4x^2}}$ & Theorem~\ref{thm: 321} \\[4pt]
    \hline
\end{tabular}
            \end{center}
            \caption{The number of permutations $\pi\in\S_n$ such that $\pi$ and $\theta(\pi)$ both avoid $\sigma$ for $\sigma\in\S_3.$}
            \label{tab:onepat} 
      \end{table}

\newpage
\subsection{Avoiding 123}

We start with the case when $\sigma=123$, showing there are actually eventually zero permutations that avoid 123 whose image under $\theta$ also avoids 123.

\begin{theorem}\label{thm:123}
            For $n \geq 11$, $t_n(123) = 0$.
        \end{theorem}

        \begin{proof}
            First note that if $\pi$ has more than three cycles in its cycle decomposition, then $\theta(\pi)$ will contain a 123 pattern given by the largest entries of any three cycles of $\pi$.
            Therefore, we can restrict to considering permutations that have at most two cycles in their cycle decomposition.


            In the case that $\pi$ has two cycles in its cycle decomposition, all the entries of the first cycle of the standard form of $\pi$ have to be decreasing, otherwise $\theta(\pi)$ will contain a $123$.
            This implies that the first cycle has length at most 3 since a cycle with four or more elements $(i_1,i_2,i_3,i_4,\ldots)$ with $i_1 > i_2 > i_3 > i_4$ gives a subsequence $i_4i_3i_2$ in $\pi$, which is a 123 pattern.

            If the second cycle is length at least 10 (or if there is only one cycle of length at least 10), since we cannot have an increasing subsequence of length $3$, the Erd\"os-Szekeres theorem implies there must be a decreasing subsequence of length $5$ among the last $9$ elements, $a_1> a_2> a_3> a_4> a_5$. 
            Consider the elements $i_1,i_2,i_3,i_4,i_5$ so that $i_j$ appears directly before $a_j$ in $\theta(\pi)$. Among these $5$ terms, by the same argument, there is a decreasing subsequence of length at least $3$, say $i_{j_1},i_{j_2},i_{j_3}$. This implies there is an increasing subsequence in $\pi$ of length $3$, namely $a_{j_3}a_{j_2}a_{j_1}$.

            This shows that $t_n(123) = 0$ for $n \geq 13$. It is straightforward to verify computationally that $t_{11}(123) = t_{12}(123) = 0$, and so the result follows.
        \end{proof}

\subsection{Avoiding 132}

Next, we consider the case when $\sigma=132,$ and find the answer in terms of Fibonacci numbers.  
\begin{theorem}\label{thm:132}
            For $n \geq1$, $t_n(132) = \displaystyle\sum_{k=0}^{n+1} F_{n+1-k} F_k$, corresponding to the generating function
            \[
           T_{132}(x) = \frac{x}{(1-x-x^2)^2}.
            \]
        \end{theorem}

        Before proving this theorem, we first state a lemma that determines where $n$ must be in the image of a permutation $\pi\in\mathcal{T}_n(132)$ under $\theta,$ and follow this up with a lemma counting the number of cyclic permutations in $\mathcal{T}_n(132)$.

        \begin{lemma}\label{lem:132-n}
            If $\pi\in\mathcal{T}_n(132)$ and $\hat\pi = \theta(\pi)$, then if $\hat\pi_i=n$, we must have $i\in\{1,n-1,n\}.$
        \end{lemma}
        \begin{proof}
            Suppose $\hat\pi_i=n$ with $2\leq i \leq n-2$. Then since $\hat\pi$ avoids 132, $\hat\pi_j>\hat\pi_k$ for all $j<i$ and $k>i$, and since $i\leq n-2$, $\hat\pi_{i+1}\neq \hat\pi_n$. This implies that $n-1$ appears to the left of $n$ in $\hat\pi$ and 1 appears to the right of $n$ in $\hat\pi$. In particular, if $\pi_1\neq n$, the standard cycle form of $\pi$ must look like 
            \[
            \cdots (n-1, \ldots, c)(n, \ldots  ,1,a, \ldots, b)
            \]
            with $a\leq b<c$. But then in $\pi$, we have $\pi_1\pi_b\pi_c = an(n-1)$, which is a 132 pattern. 

            On the other hand if $\pi_n=1$, the standard cycle form of $\pi$ must look like 
            \[
            \cdots (n-1, \ldots, c)(n,a,    \ldots, b,1)
            \] with $a\neq 1$ and so $\pi$ contains the subsequence $\pi_b\pi_c\pi_n = 1(n-1)a$ which is a 132 pattern. 
        \end{proof}

        \begin{lemma}\label{lem:132-cyc}
            The number of permutations  $\pi\in\mathcal{T}_n(132)$ so that $\theta(\pi)$ is cyclic is equal to $F_{n}$, the $n$-th Fibonacci number.
        \end{lemma}
        \begin{proof}
           Consider \cite[Theorem 5.1]{ABBGJ23}, which states that the number of cyclic permutations $\pi\in\S_n$ that avoid 213 and whose (non-standard) cycle form that begins with 1, written as $(1,\pi_1, \ldots)$, avoids 312 is equal to $F_n$. By taking the complement of the cycle form, we obtain a cycle in standard cycle form that avoids $132$. When complementing the cycle form of a cyclic permutation, we must reverse-complement the one-line form and so $\pi^{rc}$ avoids $213^{rc}=132.$ Thus the number of cyclic permutations $\pi$ that avoid 132 whose standard cycle form (and thus $\theta(\pi)$) avoids 132 is equal to $F_n.$
        \end{proof}

        We are now ready to prove the main theorem of this subsection.

        \begin{proof}[Proof of Theorem~\ref{thm:132}]
            By Lemma~\ref{lem:132-n}, we know that if $\hat\pi=\theta(\pi)$ and $\hat\pi_j=n$, then we must have that $j\in\{1,n-1,n\}$. Let us consider each case. 

            First, if $\hat\pi_1=n$, then $\pi$ is composed of only one cycle and so by Lemma~\ref{lem:132-cyc}, there are $F_n$ such permutations. Next, consider permutations with $\hat\pi_n=n$. In this case, $n$ is a fixed point of $\pi$, and so $\pi_n=n$ as well. It is clear that $n$ cannot be part of any 132 pattern in either $\pi$ or $\theta(\pi)$ and thus the number of such permutations is equal to $t_{n-1}(132).$

            Finally, consider the case when $\hat\pi_{n-1}=n$. Since $\hat\pi$ avoids 132, we must have that $\hat\pi_n=1$, and so $(n,1)$ is a 2-cycle in the cycle decomposition of $\pi$, and thus $\pi_n=1$ and $\pi_1=n.$ Notice that the $n$ and $1$ cannot be part of any 132 pattern in either $\pi$ or $\theta(\pi)$. Therefore the number of permutations with $\hat\pi_{n-1}=n$ is equal to $t_{n-2}(132).$

            This implies that $t_n(132) = t_{n-1}(132)+t_{n-2}(132) + F_n$. Together with the initial conditions that $t_1(132)=1$ and $t_2(132)=2$, it is a straightforward exercise to check that the generating function must satisfy \[T_{132}(x) = xT_{132}(x) + x^2T_{132}(x) + \frac{x}{1-x-x^2},\] and thus $t_{n}(132)$ is equal to the self-convolution of the Fibonacci numbers (OEIS A001629).
        \end{proof}

\subsection{Avoiding 213}
Now, let us consider the case when $\sigma=213$. This is one of the more complicated cases, requiring a few lemmas regarding the structure of such permutations before finding the number of permutations in terms of the Fibonacci numbers.

\begin{theorem}\label{thm:213}
            For $n \geq 2$, $t_n(213) = 2 F_{n+2} + n^2 - 6n +4$, where $F_n$ is the $n$-th Fibonacci number. Equivalently, $t_n(213)$  has the rational generating function
            \[
            T_{213}(x) =\frac{2x^5 + 9x^4 - 8x^3 -10x^2 + 13x - 4}{(x-1)^3(1-x-x^2)}.
            \]
        \end{theorem}

         \begin{lemma} \label{lem: 213-fixedIntervals}
           Given $\pi \in \mathcal{T}_n(213)$, we must have that $\pi$ is composed only of fixed points and a single cycle. Furthermore this cycle must be composed of elements in $\{n\}\cup[r, s]$ for some interval $[r,s]$.
            \end{lemma}

            \begin{proof}
            Let $\hat\pi=\theta(\pi).$ Notice that if $\hat\pi_j=n$, then the elements $\hat\pi_1\ldots\hat\pi_{j-1}$ must appear in increasing order, implying that they correspond to fixed points of $\pi.$

            Next, let us show that the elements that appear after $n$ in $\hat\pi$ form an interval in $\Z$. For the sake of contradiction, suppose there is some $2\leq a\leq n-2$ with $\pi_a=a$ so that the cycle containing $n$ includes at least one element greater than $a$ and at least one element smaller than $a$. Since $\theta(\pi)$ avoids 213, we must have that the standard cycle form of $\pi$ looks like 
            \[
            \ldots (a)\ldots (n, b_1,\ldots, b_\ell,c_1,\ldots, c_m)
            \] where $b_j>a$ for all $1\leq j\leq \ell$ and $c_k<a$ for all $1\leq k\leq m$. However, in $\pi$ we then have the subsequence $\pi_a\pi_{b_\ell}\pi_n = ac_1b_1$ which is a 213 pattern. 
     \end{proof}

            We now show that there are $2F_{n-k}-2$ permutations in $\mathcal{T}_n(213)$ which fix the elements $\{1,2,\ldots, k\}$ for any $0\leq k\leq n-3$. Since these fixed points cannot contribute to a 213 pattern in either $\pi$ or $\theta(\pi)$, counting these permutations is equivalent to counting the number of cyclic permutations of length $n-k$ that avoid $213$ and whose image under $\theta$ avoids $213$.

            \begin{lemma}\label{lem:213-cyc}
                Let $n\geq 3$. The number of cyclic permutations $\pi \in \S_n$ with $\pi$ and $\hat{\pi} = \theta(\pi)$ avoiding $213$ is $2F_{n}-2$.
            \end{lemma}

            \begin{proof}
                This lemma follows from \cite[Theorem 4.9]{ABBGJ23}. In that paper, the authors find that the number of cyclic permutations that avoid 132 and whose (non-standard) cycle form beginning with 1 avoids 231 is given by $2F_n-2.$ By complementing that cycle form, we obtain the standard cycle form that avoids $231^c=213.$ Since this corresponds to the reverse complement of $\pi,$ the one-line form of the permutation will avoid $132^{rc} = 213$.
            \end{proof}

            \begin{proof}[Proof of Theorem~\ref{thm:213}]
                By Lemma~\ref{lem: 213-fixedIntervals}, we know that $\pi\in\mathcal{T}_n(213)$ must be composed of only fixed points and a single cycle containing $n$ and elements from an interval $[r,s]$. First consider the case when $s\neq n-1$; in this case, $n-1$ must be a fixed point. However, we then must have the elements in $[r,s]$ appear in increasing or decreasing order in $\theta(\pi)$. If not, then in the cycle $(n, a_1,a_2,\ldots,a_n)$, there must be some $i$ so that $a_{i-1}<a_i>a_{i+1}$ or $a_{i-1}>a_i<a_{i+1}$.  In the first case, $\pi_{a_{i-1}}\pi_{a_i}\pi_{n-1} = a_ia_{i+1}(n-1)$ is a 213 pattern in $\pi$ and in the second case $\pi_{a_{i}}\pi_{a_{i-1}}\pi_{n-1} = a_{i+1}a_{i}(n-1)$ is a 213 pattern in $\pi$. Thus, there are exactly two possible permutations for each possible interval $[r,s]$ with $r<s<n-1,$ resulting in $2\binom{n-2}{2}$ permutations. If $r=s$, there are $n-2$ permutations, and if the interval is empty, then there is a single permutation (the identity permutation). 

                If $\pi_{n-1}\neq n-1$, then all fixed points must be in an interval in $\Z$ containing 1, and thus $\theta(\pi)$ is the direct sum of the identity permutation and a cyclic permutation in $\mathcal{T}_n(213)$. If there are $n-2$ fixed points, then we have the single permutation $\pi = 123\ldots n(n-1)$. By Lemma~\ref{lem:213-cyc}, the number of such permutations with $k$ fixed points  with $0\leq k \leq n-3$ is given by $2F_{n-k} - 2$.

                 In total we have
            \[
             2\binom{n-2}{2}+(n-2)+1+1 + \sum_{k=0}^{n-3} (2F_{n-k} - 2).
            \]
            Simplifying this expression gives us the statement of the theorem.
            \end{proof}


\subsection{Avoiding 231}

The case when $\sigma=231$ is one of the easier cases, as shown below. 

\begin{theorem}\label{thm:231}
            For $n \geq1$, $t_n(231) = 2^{n-1}$.
        \end{theorem}
            \begin{proof}
            First, let us show that there is only one cyclic permutation in $\mathcal{T}_n(231),$ namely \[\pi = n12\ldots(n-1) =(n,n-1,\ldots,1).\]
            Indeed, we must have that any cyclic permutation $\pi$ avoiding 231 must have $\pi_1=n$ since otherwise, we would have all elements appearing before $n$ forming a consecutive interval of the smallest elements of $\pi$, and thus cannot be part of the cycle containing $n$. However, then we have $\theta(\pi)$ ending in 1. Since $\theta(\pi)$ also avoids 231, $\theta(\pi)$ must be the decreasing permutation. 

            Now, let us consider any permutation $\pi\in\mathcal{T}_n(231)$ and let $\hat\pi=\theta(\pi)$. If $\hat\pi_i=n,$ then since $\hat\pi$ avoids 231, we have $\hat\pi_j<\hat\pi_k$ for all $j<i$ and $k>i$. Thus $\pi$ is a direct sum of a permutation $\pi \in \mathcal{T}_{i-1}(213)$ and the cyclic permutation of length $n-i+1.$ In particular, $\pi$ is block cyclic, where each cycle is of the form described above. It is thus enough to determine the sizes of these cycles, and thus $\mathcal{T}_n(231)$ is in bijection with compositions of $n$, implying there are $t_n(231)=2^{n-1}$.
        \end{proof}

\subsection{Avoiding 312}

The case when $\sigma=312$ is also quite straightforward, as we show all such permutations must be involutions. 
        \begin{theorem}\label{thm:312}
            For $n\geq 1$, $t_n(312) = 2^{n-1}$.
        \end{theorem}  

        \begin{proof}
            Let us first see that a permutation $\pi\in\mathcal{T}_n(312)$ cannot have any cycles of length greater than 2.  For the sake of contradiction, suppose that a permutation $\pi\in\mathcal{T}_n(312)$ does have a cycle in its cycle decomposition of length 3 or more. Then in its standard cycle form, it must be of the form $(m, a_1, a_2, \ldots, a_k)$ where $k\geq 2$ and $m>a_i$ for each $i$. Since $\theta(\pi)$ avoids 312, $a_1>a_2>\ldots>a_k$. But then, in $\pi$, we have the subsequence $\pi_{a_k}\pi_{a_1}\pi_{m}=ma_2a_1$ which is a 312 pattern.

            Therefore all cycles must be length 1 or 2, i.e., $\pi$ must be an involution. We will now show that in fact, the permutations in $\mathcal{T}_n(312)$ are exactly the 312-avoiding involutions.

            It is clear that 312-avoiding involutions are of the form $\bigoplus_{j=1}^m \delta_{d_j}$ where $(d_1,\ldots,d_m)$ is a composition of $n$ and $\delta_k$ is the decreasing permutation of length $k$. It follows that $\theta(\pi) = \bigoplus_{j=1}^m \theta(\delta_{d_j})$. It is sufficient to show that $\theta(\delta_k)$ avoids 312 for all $m.$
            Note that the standard cycle form of $\delta_k$ is 
            \[\left(\frac{k}{2}+1,\frac{k}{2}\right)\left(\frac{k}{2}+2,\frac{k}{2}-1\right) \cdots (k-1,2)(k,1) 
            \]
            if $k$ is even, and
            \[
            \left(\frac{k+1}{2}\right) \left(\frac{k+1}{2}+1,\frac{k+1}{2}-1\right) \cdots (k-1,2)(k,1)
            \]
            if $k$ is odd. Removing parentheses to get $\theta(\delta_k)$, the resulting permutation does indeed avoid 312. Since there are exactly $2^{n-1}$ involutions that avoid 312, the theorem follows.
        \end{proof}

\subsection{Avoiding 321}

In this section, we consider the most complicated case, enumerating 321-avoiding permutations whose image under the fundamental bijection also avoids 321. We give the answer in terms of its generating function. 

             \begin{theorem}\label{thm: 321} Suppose $T_{321}(x) = \sum_{n\geq 1} t_n(321)x^n$. Then 
            \[
            T_{321}(x) = \frac{2 x^2}{ 2 x(1 - x) -1 + \sqrt{1 - 4 x^2}}.
            \]
        \end{theorem}

    

We approach this by first enumerating the number of such permutations that are irreducible. To this end, we let $a_n(321)$ denote the number of irreducible permutations $\pi \in \mathcal{T}_n(321)$ and let $a(n,i)$ denote the number of such irreducible permutations with the property that $\pi_i=n$ (or equivalently, that $\hat\pi=\theta(\pi)$ ends with the element $i$).

Let us first consider the structure of irreducible permutations in $\mathcal{T}_n(321).$

\begin{lemma} \label{lem: irredEnd}
    Suppose $n\geq 2$. If $\pi \in \mathcal{T}_n(321)$ is irreducible and $\hat\pi=\theta(\pi)$, then either:
    \begin{itemize}
        \item $\hat\pi_n=n-1$ (equivalently, $\pi_{n-1}=n$), or
        \item $\hat\pi_{n-1}\hat\pi_n=ni$ (equivalently, $\pi_n=i$ and $\pi_i=n$) where $\lceil n/2 \rceil \leq i \leq n-2$.
    \end{itemize}
\end{lemma}

\begin{proof}
Since $\pi$ is irreducible and $n\geq 2$, we must have $\pi_n<n$, and so $\hat\pi_n<n$, so for the sake of contradiction, assume $\pi_n<n-2$ and is not of the form described in the second bullet point in the statement of the theorem. First, suppose 
$\hat{\pi} = \theta(\pi)$ ends in $ni_1i_2\ldots i_k$ where $k \geq 2$  with $i_1 < i_2 < \cdots < i_k$ and $i_k \not= n-1$. There are two cases to consider: either $n-1$ is a fixed point in $\pi$ or not. If $n-1$ is fixed, then $\pi$ contains a $321$ given by $n(n-1)i_1$. If $n-1$ is not fixed, then it appears as part of a cycle $(n-1, j_1, \ldots, j_\ell)$. So $\hat{\pi}$ ends with the sequence $(n-1) j_1 \ldots j_\ell ni_1i_2\ldots i_k$. Since $\hat{\pi}$ avoids $321$, $j_1 <  \cdots < j_\ell < i_1 < \cdots < i_k$. However,  $\pi$ then contains $\pi_{j_\ell}\pi_{i_1}\pi_n=(n-1) i_2 i_1$ as a subsequence, which is a 321 pattern.

Now suppose $\hat{\pi}$ ends in $ni$ with $i < \lceil \tfrac{n}{2} \rceil$. Then $\pi_i = n$ and $\pi_n = i$, but there are more than $\tfrac{n}{2}$ elements between $i$ and $n$, and less than $\tfrac{n}{2}$ numbers smaller than $i$. This means there will be some $j>i$ appearing after $n$ in $\pi$ and thus $\pi$ will necessarily contain an occurrence of $321$, namely $nji$.
\end{proof}

Now, let us find the value of $a(n,i)$ for all possible values $\lceil n/2 \rceil \leq i \leq n-1$, starting with the special case when $i=n-1.$

\begin{lemma} \label{lem: irred_n-1}
    For $n \geq 3$, $a(n,n-1) = a_{n-1}(321).$ 
\end{lemma}

\begin{proof}
    Given an irreducible permutation $\pi \in \mathcal{T}_{n-1}(321)$, with $\pi = \pi_1 \ldots \pi_{n-2}\pi_{n-1}$, we can obtain a permutation $\pi' = \pi_1 \ldots \pi_{n-2} n \pi_{n-1}$ in $\mathcal{T}_{n}(321)$, that is also irreducible with the property that $\pi_{n-1}=n$. Note that this corresponds to appending $n-1$ to the end of the image under $\theta$ (adjusting the other values of $\theta(\pi)$ as necessary).
    Furthermore, all such irreducible permutations in $\mathcal{T}_n(321)$ can be obtained this way since this process is clearly invertible. 
\end{proof}


\begin{lemma} \label{lem: irred_ni}
    For $n \geq 6$ and $\lceil n/2 \rceil \leq i < n-2$, 
    $a(n,i) = \displaystyle\binom{i-2}{\lceil \tfrac{n-2}{2} \rceil - 1}$.
\end{lemma}

\begin{proof}
    Suppose $\pi \in \mathcal{T}_n(321)$ with $\hat\pi=\theta(\pi)$ ending in $\hat\pi_{n-1}\hat\pi_n=ni$ for $\lceil n/2 \rceil \leq i < n-2$, so that $\pi = \pi_1 \ldots \pi_{i-1} n \pi_{i+1} \ldots \pi_{n-1} i$. Let us obtain a new permutation $\pi'$ in $S_{n-2}$ by removing $n$ and $i$ from $\pi$ and subtracting one from every element of $\pi$ that is greater than $i$. For example, if $\pi=247189356$, then $\pi'=2461735$. 
    Notice that $\pi'$ still avoids $321$ since $\pi$ does and also, $\pi'$ is still irreducible since $\pi$ is and since $\pi_j<i$ for all $i<j<n$. Furthermore, $\hat\pi'=\theta(\pi')$ also avoids $321$ since we can obtain $\hat\pi'$ from $\hat\pi$ by removing $ni$ from the end of $\hat\pi$ and subtracting one from all the elements larger than $i$. For example, given $\pi=247189356$, and $\pi'=2461735$, we see that $\hat\pi=412738596$ and $\hat\pi'=4126375$. This works since $n$ and $i$ formed a 2-cycle in $\pi$ and so removing those two elements from $\pi$ corresponds to removing that cycle from the standard cycle form of $\pi.$ 

    This means $\pi' \in \mathcal{T}_{n-2}(321)$ and is irreducible. By Lemma \ref{lem: irredEnd}, $\hat\pi'=\theta(\pi')$ can either end in $n-3$ or $(n-2)j$ for some $\lceil \frac{n-2}{2} \rceil \leq j \leq n-4$.

    Let us first see that $\hat\pi'$ does not end in $n-3.$ If it did, that would imply that $\pi_{n-1}=n-2$. But then $\pi$ contains $\pi_i\pi_{n-1}\pi_n=n (n-2) i$, which is a 321 pattern.
    If $\hat\pi'$ ends in $(n-2) j$ for $i \leq j \leq n-4$ this means $\pi$ contains the cycle $(n-1,j+1)$. But then $\pi$ contains $\pi_{i}\pi_{j+1}\pi_n=n (n-1)i$ which is a $321$ pattern.

    So we know that $\hat\pi'$ ends with $(n-2)j$ for $\lceil \frac{n-2}{2} \rceil \leq j \leq i-1$. 
    In fact, any such permutation in $\mathcal{T}_{n-2}(321)$ with these properties can be obtained from a permutation in $\mathcal{T}_n(321)$ with the property that $\pi_i=n$ and $\pi_n=i$ for $\lceil n/2\rceil\leq i<n-2$. Indeed, starting with such a permutation in $\tau\in\mathcal{T}_{n-2}(321)$, and adding the two-cycle $(n,i)$ (adjusting the other elements of $\tau$ accordingly), must still be irreducible and must still avoid 321 since in $\tau$, the elements between $n-2$ and $j$ are each less than $j$, thus less than $i$, and appear in increasing order. 
    




    Thus, we have the established the recurrence
\[
    a(n,i) = \sum_{j=\big\lceil \tfrac{n-2}{2} \big\rceil}^{i-1} a(n-2,j).
    \]
    Let us now see that $a(n,i) = \binom{i-2}{\lceil \tfrac{n-2}{2} \rceil - 1}$ for any $\lceil n/2 \rceil \leq i < n-2$ by strong induction on $n$. The base cases when $n=6,7$ are easy to verify.
    Suppose that our inductive hypothesis holds for all $k\leq n-1$, and we will show that $a(n,i) = \binom{i-2}{\lceil \tfrac{n-2}{2} \rceil - 1}$ for $\lceil n/2 \rceil \leq i < n-2$. Indeed,
    \[
    a(n,i) = \sum_{j=\lceil \tfrac{n-2}{2} \rceil}^{i-1} a(n-2,j) = \sum_{j=\lceil \tfrac{n-2}{2} \rceil}^{i-1} \binom{j-2}{\lceil \frac{n-4}{2} \rceil -1} = \sum_{j=\lceil \tfrac{n-4}{2} \rceil -1}^{i-3} \binom{j}{\lceil\frac{n-4}{2} \rceil -1} = \binom{i-2}{\lceil \frac{n-2}{2} \rceil -1}.
    \]
    The first equality is the established recursion, the second equality follows by strong induction, 
    the third equality is re-indexing and the final equality is a well known identity for binomial coefficients.
\end{proof}

With this lemma established, we handle the final remaining case.

\begin{lemma}\label{lem: ir n-2}
    For $n\geq 6$, $a(n,n-2)=\displaystyle\binom{n-4}{\lceil \tfrac{n-2}{2} \rceil - 1}$
\end{lemma}

\begin{proof}
Let us suppose $\pi\in\mathcal{T}_n(321)$ with $\hat\pi=\theta(\pi)$ ending in $n(n-2)$, and so $\pi$ ends in $n\pi_{n-1}(n-2).$ Note that we must not have $\pi_{n-1}=n-1$ since this would give us 321 pattern, but that any other possible value of $\pi_{n-1}$ would not result in a 321 pattern. Indeed, we can append the 2-cycle $(n,n-2)$ to any permutation in $\tau\in\mathcal{T}_{n-2}(321)$ that does not end with $n-1$ (adjusting the remaining values of $\tau$ accordingly), and obtain all such permutations in $\mathcal{T}_n(321)$. Thus, there are $\sum_{j=\lceil \tfrac{n-2}{2} \rceil}^{n-3} a(n-2,j)$, which by Lemma~\ref{lem: irred_ni}, implies the result.
\end{proof}

Since we have now found $a(n,i)$ for all possible values of $i$, we can now determine $a_n(321)$, the number of irreducible permutations in $\mathcal{T}_n(321).$

\begin{theorem}\label{thm: 321-Irred}
    For $n \geq 2$,
    \[
        a_n(321) = \binom{n-2}{\lfloor \frac{n-2}{2} \rfloor},
    \] with the corresponding generating function 
     \[
    A(x) = \frac{x-x^3c(x^2)}{1-x-x^2c(x^2)},
    \]
    where $c(x)$ is the generating function for the Catalan numbers.
\end{theorem}

\begin{proof}
    We will proceed by induction on $n$. The cases $a_i(321)$ for $2\leq i \leq 6$ are easy to verify. Assume the result holds for $n-1 \geq 6$. By Lemmas \ref{lem: irredEnd} and \ref{lem: irred_n-1}, $a_n = a_{n-1} + \sum_{i=\lceil n/2 \rceil}^{n-2} a(n,i)$.

    Using Lemma \ref{lem: irred_ni} and \ref{lem: ir n-2}, we see that
    \[
    \sum_{i=\lceil n/2 \rceil}^{n-2} a(n,i) =  \sum_{i=\lceil n/2 \rceil}^{n-2} \binom{i-2}{\lceil\frac{n-2}{2}\rceil -1}=\sum_{k=0}^{\lfloor \tfrac{n}{2} \rfloor -2} \binom{\lceil \tfrac{n-2}{2} \rceil- 1 + k}{\lceil \tfrac{n-2}{2} \rceil -1} = \binom{n-3}{\lceil \tfrac{n-2}{2} \rceil}.
    \]
    By induction, $a_{n-1}$ is equal to $\binom{n-3}{\lfloor \tfrac{n-3}{2} \rfloor}$ and thus,
    \[
    a_n = a_{n-1} + \sum_{i=\lceil n/2 \rceil}^{n-2} a(n,i) = \binom{n-3}{\lfloor \tfrac{n-3}{2} \rfloor} + \binom{n-3}{\lceil \tfrac{n-2}{2} \rceil} = \binom{n-2}{\lfloor \tfrac{n-2}{2} \rfloor}.
    \]
    Finally, since the binomial coefficient $\binom{n-2}{\lfloor \tfrac{n-2}{2} \rfloor}$ is equal to the $\lfloor \tfrac{n-2}{2} \rfloor$-th Catalan number, the generating function in the statement of the theorem can be obtained. 
\end{proof}

Using the generating function for the number of irreducible permutations in $\mathcal{T}_n(321)$, we can obtain the generating function for all permutations in $\mathcal{T}_n(321)$.
\begin{proof}[Proof of Theorem \ref{thm: 321}]
    By Theorem \ref{thm: 321-Irred},  the generating function for irreducible permutations in $\mathcal{T}_n(321)$ is
    \[
    A(x) = \frac{x-x^3c(x^2)}{1-x-x^2c(x^2)},
    \]
    where $c(x)$ is the generating function for the Catalan numbers. All permutations in $\mathcal{T}_n(321)$ are of the form $\bigoplus_{j=1}^m \tau_{d_j}$ where $(d_1,\dots,d_m)$ is a composition of $n$ and $\tau_k$ is an irreducible permutation in $\mathcal{T}_k(321)$. Furthermore, for any such choice of $\tau_{d_j}$, $\bigoplus_{j=1}^m \tau_{d_j}$ is in $\mathcal{T}_n(321)$. It follows that the generating function for $\mathcal{T}_n(321)$ is $\frac{A(x)}{1-A(x)}$, which simplifies to
    \[
    T_{321}(x) = \frac{2x^2}{2x(1-x)-1+\sqrt{1-4x^2}}.
    \]
\end{proof}




\section{Avoidance in higher iterations of $\theta$} \label{sec: HigherIterations}

In this section, we consider $\sigma$-avoiding permutations that still avoid $\sigma$ under more than one iteration of $\theta.$ Let us denote by $\mathcal{T}^k_n(\sigma)$ the set of permutations $\pi\in\S_n$ so that $\pi, \theta(\pi), \theta^2(\pi),\ldots,\theta^k(\pi)$ all avoid the pattern $\sigma$. Notice that $\mathcal{T}^1_n(\sigma) = \mathcal{T}_n(\sigma).$ We will denote by $t^k_n(\sigma):=|\mathcal{T}^k_n(\sigma)|$.

For example, consider the permutation 
$\pi = 134579862$, which avoids the pattern $213.$ Notice that its standard cycle form is $\pi=(1)(9,2,3,4,5,7,8,6)$, so $\theta(\pi) = 192345786,$ which also avoids 213 and has cycle form $\theta(\pi) = (1)(7)(8)(9, 6, 5, 4, 3, 2)$. Therefore, we get $\theta^2(\pi) = 178965432$, which also avoids 213. In this case we would say that the original permutation $\pi = 134579862$ is in the set $\mathcal{T}^2(213).$ However, $\pi\not\in\mathcal{T}^3(213)$ since $\theta^3(\pi) = 1564839274$ contains a 213 pattern.

We first consider permutations in the set $\mathcal{T}_n^k(213),$ as summarized by Table~\ref{tab:231-higher}. Notice that for small values of $k$, we get varying results, but for all $k\geq 5,$ there are exactly 7 permutations with the property that $\pi, \theta(\pi), \theta^2(\pi),\ldots, \theta^k(\pi)$ all avoid 213. 

\renewcommand{\arraystretch}{2}
\begin{table}[htp]
    \centering
    \begin{tabular}{c|c}
         $k$ & $t_n^k(213)$\\ \hline
        1& $2F_{n+2}+n^2-6n+4$  \\
           2 & $\displaystyle\binom{n+1}{2}$  \\
           3 & $2n+1$  \\
           4 & $n+4$  \\
           $\geq 5$ & 7  
    \end{tabular}
    \caption{The number of permutations so that $\pi, \theta(\pi), \theta^2(\pi),\ldots, \theta^k(\pi)$ all avoid 213 for different values of $k.$ Notice the first result follows from Theorem~\ref{thm:213}, the second from Theorem~\ref{thm:213-t2}, and the remaining results follow from Theorem~\ref{thm:213-tk}.}
    \label{tab:231-higher}
\end{table}

\begin{theorem}\label{thm:213-t2}
    For $n\geq 4$, $t_n^2(213) =\binom{n+1}{2}$.
\end{theorem}

\begin{proof}
    First, notice that if $(1)$ is a fixed point, it will also be a fixed point of $\theta^k(\pi)$ for any exponent $k\geq 0$. Since 1 appears in the first position of $\theta^k(\pi)$, it cannot be part of any 213 pattern, and thus there are $t^2_{n-1}(213)$ permutations in $\mathcal{T}^2_n(213)$ that have 1 as a fixed point.

    Let us now consider those permutations that do not have $1$ as a fixed point. Suppose $\pi$ is such a permutation. We first claim that if $n\geq 5$, then either
    \begin{itemize}
     \item $\pi$ is the permutation $n234\ldots (n-1)1$, which implies $\theta(\pi)=23\ldots n1$ and $\theta^2(\pi) =n123\ldots n-1$, all of which avoid 213, or
     \item $\pi$ is cyclic.
    \end{itemize}
    Let us assume that $\pi$ is not the permutation listed in the first bullet above, and let us see why $\pi$ must indeed be cyclic. For the sake of contradiction, suppose not. Then, by the proof of Theorem~\ref{thm:213}, since 1 is not a fixed point and $\theta(\pi)$ avoids 213, the cycle form of $\pi$ must be either 
    \[\pi = (s)(s+1)\ldots(n-1)(n,s-1,\ldots, 2,1)\] or \[\pi = (s)(s+1)\ldots(n-1)(n,1,2,\ldots, s-1)\] for some value of $s\in [3,n-1]$.  Let us write $\hat\pi=\theta(\pi)$. Then since neither $1$ nor $n-1$ is a fixed point of $\hat\pi$, it must be cyclic by Lemma~\ref{lem: 213-fixedIntervals}. Thus, if $s<n-1,$ the cycle form of $\hat\pi$ must be either of the form
    \[\hat\pi = (n,1,s\ldots, 2,s+1,\ldots)\]
    or \[\hat\pi = (n,s-1,\ldots, 1,s,\ldots)\]
    both of which have a 213 pattern. If $s=n-1$, then either $\pi=(n-1)(n,n-2,\ldots,1)$ and so $\hat\pi$ has $(n,1,n-1,2)$ as a cycle, contradicting that $\hat\pi$ is cyclic, or $\pi=(n-1)(n,1,2,\ldots,n-2)$ and so $\hat\pi = (n,n-2,\ldots, 1,n-1,\ldots)$ which has a 213 pattern.
    Thus we have shown that $\pi$ is cyclic. 

    Now, since $\pi$ is cyclic, we must have that $\hat\pi_1=n$. Let us consider two cases: where $\hat\pi_{n-1}=n-1$, or $\hat\pi_{n-1}\neq n-1$.

    In the first case where $\hat\pi_{n-1}=n-1$, we must have $\hat\pi = n12\ldots (s-1)(s+1)\ldots (n-1)s$ for some $s$ since $\hat\pi$ avoids 213. But then in standard cycle form, $\hat\pi = (s+1)(s+1)\ldots (n-1)(n,s,s-1,\ldots, 2,1)$, and so $\theta(\hat\pi)$ avoids 213. Notice there are $n-2$ of these permutations for the $n-2$ possible values of $s.$

    In the second case where $\hat\pi_{n-1}\neq n-1$, we must have that $\hat\pi$ is cyclic by Lemma~\ref{lem: 213-fixedIntervals} since neither 1 nor $n-1$ are fixed points. We claim that in this case $\theta^2(\pi) =\theta(\hat\pi) = n(n-1)\ldots21$. For the sake of contradiction, suppose not. Then there is some $k$ so that \[\hat\pi=(n,b_1,\ldots,b_r, k+1, a_1,\ldots,a_t,k,k-1,\ldots,2,1).\] Thus, if $r\geq 1,$ since $b_i>a_j$ for all $i$ and $j$ since $\theta(\hat\pi)$ avoids 213, we have that in one-line notation, $\hat\pi = 23\ldots(k-1)(k+1)\ldots k\ldots b_1$ contains the 213 pattern  $(k+1)kb_1.$ If $r=0$, and so \[\hat\pi=(n,k+1, a_1,\ldots,a_t,k,k-1,\ldots,2,1),\] then it must be the case that $\hat\pi=n12\ldots(k-1)a_1\ldots km\ldots (k+1)$  where $m=\pi_{a_t+1}<a_1$ since $\hat\pi$ avoids 213. But then,\[\pi = (n,1,2,\ldots,(k-1),a_1,\ldots,k,m,\ldots ,(k+1)) = 2\ldots (k-1)a_1 m n \ldots (k+1)\ldots 1.\]
    where $a_1mn$ is a 213 pattern. 

    Finally, we can say that the only permutations that don't have 1 as a fixed point are as follows:
    \begin{itemize}
        \item $\theta(\pi) = n123\ldots (n-1)$,
        \item $\theta(\pi) = n(n-1)\ldots 21$, or
        \item $\theta(\pi) = n123\ldots (s-1)(s+1)\ldots (n-1)s$ for some $1\leq s\leq n-2$.
    \end{itemize}  This implies that the number of permutations in $\mathcal{T}_n^2(213)$ satisfies the recurrence \[t_n^2(213)=t_{n-1}^2(213)+n,\]
    and so the result follows. 
\end{proof}

\begin{theorem}\label{thm:213-tk}
    For $n\geq 5$, $t_n^3(213) =2n+1$, $t_n^4(213) =n+4$, and $t_n^k(213) =7$ for $k\geq 5$.
\end{theorem}

\begin{proof}
    As in the previous proof, for any $k$, it is clear that there are $t_{n-1}^k(213)$ permutations in $\mathcal{T}_n^k(213)$ that have 1 as a fixed point. So let us consider the permutations that do not have 1 as a fixed point.
    By the proof to Theorem~\ref{thm:213-t2}, we know that the permutations in $\mathcal{T}^2_n(213)$ that don't have 1 as a fixed point have the property that either:
    \begin{itemize}
        \item $\pi =(2)(3)\ldots(n-1)(n,1)$ with $\theta^2(\pi) =n123\ldots(n-1)$,
        \item $\pi = (n,1,2,3,\ldots, n-1)$ with $\theta^2(\pi) = n(n-1)\ldots21,$ or
        \item $\pi =(n,1,2,\ldots,s-2,s,\ldots,n-1,s-1)$ with $\theta^2(\pi) =s(s+1)\ldots n(s-1)(s-2)\ldots 21$ for any $2\leq s\leq n-1$.
    \end{itemize} 
    Notice that in the first case, $\theta^3(\pi) = n(n-1)\ldots321$, which avoids 213. Notice that in the second and third case, $n-1$ is clearly not a fixed point of  $\theta^2(\pi)$, and so for $\theta^3(\pi)$ to avoid 213, by Lemma~\ref{lem: 213-fixedIntervals}, you would need the permutation to be cyclic. However, for $n\geq 5$, these are only cyclic when $\pi =(n,2,3,\ldots,n-1,1)$ with $\theta^2(\pi) =23\ldots n1$ and thus $\theta^3(\pi) = n123\ldots n-1$, or when $n$ is even and $\pi = (n,1,3,4,\ldots,n-1,2)$ and $\theta^2(\pi) = 234\ldots n21$, and so $\theta^3(\pi) = (n,1,3,5,\ldots,2,4,6,\ldots )$ which does contain a 213 pattern. Therefore, for $n\geq 5$, there are only 2 permutations in $\mathcal{T}_n^3(213)$ that don't have 1 as a fixed point. Together with the fact that $t_4^3(213)=9$, the first result follows for $n\geq 4.$

    Now, of these two permutations, only the permutation $\pi = (n,2,3\ldots,n-1,1)$ has the property that $\theta^4(\pi) = n(n-1)\ldots 321$ also avoids 213. However $\theta^5(\pi)$ does not avoid 213. These two facts imply that for $n\geq 5$, we have that $t_n^4(213)=t_{n-1}^4(213)+1$ and that $t_n^k(213)= t_{n-1}^k(213)+0$ for $k\geq 5$, which together with the initial conditions imply the other two results.
\end{proof}

Before considering other patterns $\sigma\in\S_3$, let us consider the lemma below regarding the permutations fixed under $\theta$. This is a well-known result, but we provide a proof here for completeness. 
\begin{lemma}\label{lem:fixed}
    For $n\geq 2$, there are $F_{n+1}$ permutations fixed by the fundamental bijection. These permutations are exactly involutions where every 2-cycle is composed of a consecutive pair of elements.    
\end{lemma}

\begin{proof}
Let $s_n$ denote the number of permutations $\pi\in\S_n$ so that $\theta(\pi)=\pi$.
    Since it is easily checked $s_1=1$ and $s_2=2$, it is enough to check that $s_n=s_{n-1}+s_{n-2}.$ Given any such permutation, if $n$ is a fixed point, then both $\pi$ and $\theta(\pi)$ end with $n$. After removing that 1-cycle, you end up with a permutation in $\S_{n-1}$ fixed under $\theta$. If $n$ is not a fixed point, then the last cycle of $\pi$ is of the form $(n,\ldots, j)$ for some $j$ and so $\pi_j=n$. Since $\pi=\theta(\pi)$, it must also be the case that $\pi_n=j$, and thus $n$ is part of a 2-cycle $(n,j).$ But then $n$ must appear in the position $n-1$ of $\pi$, and so we must have $j=n-1$. Removing this 2-cycle, we obtain a permutation in $\S_{n-2}$ fixed under $\theta.$ The form of the permutations clearly results from this recursive process.
\end{proof}

The permutations described above clearly avoid the patterns 231, 312, and 321. In the theorem below, we find that these are actually the only permutations in $\mathcal{T}_n^k(\sigma)$ for any $\sigma\in \{231,312,321\}$ for $k\geq 2.$

\begin{theorem} \label{thm:kPowerFib}
    For $n\geq1$, $k\geq 2$, and $\sigma\in\{231,312,321\}$, $t_n^k(\sigma) =F_{n+1}$.
\end{theorem}

\begin{proof}
    Since the elements described by Lemma~\ref{lem:fixed} are fixed by $\theta^k$ for any $k\geq 1$ and avoid the patterns 231, 312, and 321, it is clear that $\mathcal{T}^k_n(\sigma)$ must contain these permutations for each $\sigma \in \{231,312,321\}$, and thus $t_n^k(\sigma)\geq F_{n+1}.$ It remains to show that these are indeed all the permutations in $\mathcal{T}_n^k(\sigma)$. If we show this is the case for $k=2$, it must also follow for all $k\geq 3$ since by definition, $\mathcal{T}_n^i(\sigma)\subseteq\mathcal{T}_n^j$ for any $i\geq j$.
    
    This theorem is straightforward for the cases when $\sigma=231$ or 312. Indeed, by Theorem~\ref{thm:231}, we know that for any $\pi \in \mathcal{T}^2_n\subseteq\mathcal{T}_n$, we must have that $\pi$ is block cyclic of the form $\bigoplus_k \epsilon_{d_k}$ where $\epsilon_{d}=d12\ldots(d-1)=(d,d-1,\ldots,2,1)$. Since for any $\pi\in\mathcal{T}^2_n$, we must have that $\theta(\pi)$ is also of this form, it must be the case that $\pi$ is composed only of 1-cycles and 2-cycles. Since it is block cyclic, the theorem follows. Similarly for $\sigma=312$, from Theorem~\ref{thm:312}, we know the permutations must be of the form $\bigoplus_k \delta_{d_k}$ where $\delta_{d}=d(d-1)\ldots 21=\cdots(d-1,2)(d,1)$, and since $\theta(\pi)$ must also be of this form, the blocks must be size 1 or 2. 

    The final case to consider is when $\sigma=321$. In Theorem \ref{thm: 321} we argue that every permutation $\pi \in \mathcal{T}_n(321)$, and therefore in $\mathcal{T}^2_n(321)$, is of the form $\bigoplus_{j=1}^m \tau_{d_j}$, where $\tau_d$ is an irreducible permutation in $\mathcal{T}_d(321)$. If $\pi \in \mathcal{T}^2_n(321)$, this also must be true of $\theta(\pi)$, so it suffices to prove that the number of irreducible $\tau \in \mathcal{T}^2_n(321)$ is zero when $n \geq 3$.


    Suppose $\tau \in \mathcal{T}_n^2(321)$ is irreducible with $n \geq 3$, and consider the cycle in $\tau$ that contains $1$.
    Notice that $1$ cannot be fixed, or be in the transposition $(2,1)$ in $\tau$ since $\tau$ is irreducible and $n \geq 3$. 
    Furthermore, we claim that 1 must appear in the first cycle in the standard cycle form of $\tau$. If not, then since $\theta(\tau)$ avoids 321, all elements before 1 in $\theta(\tau)$ are increasing, and so it must be the case that all cycles preceding the element containing 1 in $\tau$ are fixed points. But this implies there is some $1<r<s$ with $\tau_r=r$ and $\tau_s=1$. Since $\tau$ is irreducible, there is an element preceding $r$ that is greater than $r$, which gives us an 321 pattern in the one-line notation of $\tau.$ 

    Thus $1$ is in the first cycle of $\tau$, which must be of the form $(k,1,2,\ldots, j)$ for some $j<k$ and $k\geq 3.$ If $j\geq 2$, then since $\theta^2(\tau)$ contains the subsequence 321 since in $\hat\tau=\theta(\tau)$, $\hat\tau_3=2$ and $\hat\tau_2=1.$
    Similarly, if $j=1$, then 2 must appear in a position $m$ greater than 2, so $m21$ is a subsequence of $\theta^2(\tau)$. Therefore, there are no irreducible permutations in $\mathcal{T}_n^2(321)$ with $n\geq 3.$
\end{proof}


Since by Theorem~\ref{thm:123}, the number of permutations avoiding 123 whose image under $\theta$ also avoids 123 is eventually zero, we must have $t_n^k(123)=0$ for large enough $n.$ Therefore the only remaining case to consider is 132. This case seems pretty complicated, and we leave the enumeration of $\mathcal{T}^k(132)$ for $k\geq 2$ as a conjecture.
\begin{conjecture}
    For $n\geq2$,
    \[
    t_n^2(132) = \begin{cases}
        k^3 + 3 k^2 + 2 k - 1 & n=3k \\ 
        k^3 + 4 k^2 + 4k  & n=3k+1\\
        k^3 + 5 k^2 + 7 k + 2 & n=3k+2
    \end{cases}
    \]
    The conjectured values for  $k\geq 3$ are found in the table below.
    \begin{center}
    \begin{tabular}{c|c}
    $k$ & $t_n^k(132)$ \\ \hline
           $3$ & $3n-4$  \\
           $4$ & $2n-1$ \\
           $5$ & $n+2$  \\ 
           $\geq6$ & $5$  \\ 
    \end{tabular}
    \end{center}
\end{conjecture}

\section{Avoidance for fixed-order permutations under $\theta$-action} \label{sec: fixedOrder}

Notice that $\Z$ acts on $\S_n$ by $k\star\pi = \theta^k(\pi)$ (since $\theta$ is a bijection, we let $\theta^{-1}$ be the inverse of $\theta$.) We will say that $\pi$ is fixed by $k$ if $k\star\pi=\pi.$ Let $\mathcal{F}_n^k(\sigma)$ be the set of permutations fixed by $k$ that avoid $\sigma$, let $f_n^k(\sigma) = |\mathcal{F}_n^k(\sigma)|$, and let the generating function for these values be denoted by $F^k_\sigma(x) = \sum_{n\geq 0} f_n^k(\sigma)x^n$. Surprisingly, determining these $f_n^k(\sigma)$ yields some interesting results.

Let us first consider when $k=1,$ in which case $\pi=\theta(\pi).$ Recall that by Lemma~\ref{lem:fixed}, permutations that are fixed under the fundamental bijection are composed of fixed points and transpositions of consecutive elements.

\begin{theorem}
    For $n\geq 5$, 
    \[f_n^1(\sigma) = \begin{cases}
        F_{n+1} & \sigma\in\{231,312,321\},\\
        2 & \sigma\in\{132,213\},\\
        0 & \sigma\in\{123\}.
    \end{cases}\]
\end{theorem}

\begin{proof}
    The permutations described in Lemma~\ref{lem:fixed} clearly avoid 231, 312, and 321, and so there are $F_{n+1}$ permutations fixed by $\theta$ and avoiding one of these patterns.
    In order to avoid the pattern 132, a permutation fixed by $\theta$ must be of the form $1234\ldots n$ or $2134\ldots n$. If a permutation $\pi$ of the correct form had any other transposition $(i+1,i)$ for $i>2$, then $1(i+1)i$ would be a 132 pattern in $\pi$. Similarly, any permutation fixed by $\theta$ that avoids 213 must be of the form $12\ldots n$ or $12\ldots (n-2)n(n-1)$. Finally, any permutation fixed by $\theta$ with at least five elements must have at least 3 cycles, and taking any single element from each cycle will result in a 123 pattern.
\end{proof}

We next consider the permutations that avoid some $\sigma \in S_3$ that are fixed by 2 under this group action. Though understanding permutations fixed by 2 in general appears to be a more difficult question (see Section~\ref{sec: Conclusion}), when restricting to pattern-avoiding permutations, the answers turn out to be quite nice.  Note that for any permutation $\pi$ with $\pi=\theta^2(\pi)$, the standard cycle form of $\hat\pi=\theta(\pi)$ must appear in the same order as $\pi$ itself since $\theta(\hat\pi)=\pi;$ in other words, $\hat\pi=\theta(\pi)=\theta^{-1}(\pi).$

\begin{theorem}
    For $\sigma\in\{231,312\},$ we have 
    \[F_{\sigma}^2(x) = \frac{1}{1 - x - x^2 - x^4},\]
\end{theorem}

\begin{proof}
    First recall that by Lemma~\ref{lem: thetaCommutes}, we know that $\theta(\tau\oplus\rho) = \theta(\tau)\oplus\theta(\rho).$ Since $\pi$ avoids 312, $\pi = \tau\oplus\rho$ where both $\tau$ and $\rho$ avoid 312 and $\tau$ ends in 1. Therefore, it will be enough to determine which permutation that lie in $\mathcal{F}_n^2(312)$ also end with 1.

    Suppose $\tau$ is such a permutation and $\hat\tau=\theta(\tau)$. It is easy to check that there is one such permutation for $n\in\{1,2,4\}$ and none for $n=3$. Let us assume $n>4$ and show there are no such permutations in $\mathcal{F}_n^3(312)$ that end in 1. First, note that if $\tau=\tau_1\ldots\tau_{n-1}1$, since $\theta(\hat\tau)=\tau,$ we must have that $\theta(\hat\tau)$ ends in 1 and thus  $\hat\tau_1=n$ (implying $\tau$ is cyclic) and $\hat\tau_2=1$. 
    
    Notice also that since $\hat\tau_2=1$ and $\tau_n=1$, we must have that $\tau_{n-1} = 2$, which in turn, implies that $2$ follows $n-1$ in $\hat\tau$. First suppose that 2 is in position $j<n-1$ in $\hat\tau$. Now, since $\tau=\theta(\hat\tau)$, we must have that one of $j2(j-1)$, $n(j-1)j$, or $(n-1)(j-1)j$ is a subsequence of $\tau$, all of which are 312 patterns. Therefore, $j\geq n-1$. If $j=n-1,$ we must have $n(n-2)(n-1)$ is a subsequence of $\tau$, and so we must have $2$ in the $n$-th position in $\hat\tau$. However, this implies that $\hat\tau=n1\ldots (n-1)2$, and so $\theta(\hat\tau)=\ldots(n-1) n2\ldots 1$, but also, $\tau=\tau_1n\ldots 21$, which is impossible if $n>4$.

    Finally, since for any $\pi = \tau\oplus\rho\in\mathcal{F}^2_n(312)$ with $\tau$ ending in 1, we must have that $\tau$ is size 1, 2, or 4. It follows that $f_n^2(312)= f_{n-1}^2(312)+f_{n-2}^2(312)+f_{n-4}^2(312)$, and the result follows. 

    An almost identical argument works for $231$ except in this case since $\pi$ avoids $231$, $\pi = \tau \oplus \rho$ where both $\tau$ and $\rho$ avoid $231$ and $\rho \in \S_m$ starts with its largest element, $m$. So it is enough to determine which permutations that lie in $\mathcal{F}_m^2(231)$ also begin with $m$. One can similarly argue that there is only one such permutation for $m \in \{1,2,4\}$.
\end{proof}


Now let us see that permutations avoiding 321 that are fixed by $\theta^2$ are exactly those that are fixed by $\theta$ itself.


\begin{theorem}
    For $n \geq 2$, $f_n^2(321) = F_{n+1}$. 
\end{theorem}

\begin{proof}
    By Lemma \ref{lem:fixed} $f_n^2(321) \geq F_{n+1}$, so it suffices to show that the permutations described by that lemma are exactly those in $\mathcal{F}_n^2(321)$. We argue in the same spirit as the proof of Theorem \ref{thm:kPowerFib}; that is, we show that if $n\geq 3$, a permutation in $\mathcal{F}_n^2(321)$ cannot be irreducible. For the sake of contradiction, suppose $\tau \in \mathcal{F}_n^2(321)$ is irreducible with $n \geq 3$ and let $\hat\tau=\theta(\tau)$. Consider the cycle in the standard cycle decomposition $\hat\tau$ containing the element $n$. Since $\theta(\hat\tau) = \tau$ avoids 321, this cycle must be of the form $(n,a_1, a_2, \ldots, a_k)$ where $a_1<a_2<\ldots <a_k$. Note that this also implies the one-line form of $\tau$ ends with $na_1a_2\ldots a_k.$ Since $\tau_n=a_k$, $a_k$ immediately follows $n$ in $\hat\tau$. However, if $k>1,$ we would have $\hat\tau_{a_k}=n$ and $\hat\tau_{a_{k-1}}=a_k$ with $a_{k-1}<a_k$, which is a contradiction, and so $k=1.$ But then, $\tau_{n-1}=n$ and so $\hat\tau_{n}=n-1$, which implies $a_1=n-1$. Since the final cycle is of the form $(n,n-1)$ and $n\geq 3$, this permutation is clearly not irreducible.
\end{proof}

Interestingly, for the remaining values of $\sigma$, the number of permutations fixed by 2 under this action avoiding the given pattern is eventually constant. Before addressing these cases, we state the following lemma about all permutations $\pi\in\S_n$ with the property that $\pi=\theta^2(\pi).$

\begin{lemma} \label{lem: order2} For $\pi\in\S_n$ with $n\geq 3$ and $\hat\pi=\theta(\pi)$, if $\pi=\theta^2(\pi)$ and $\pi_n\neq n$, then
    \begin{itemize}
    \item $\pi_j=n$ implies that 
    $\pi_{j+1}=j$, and 
    \item if $\pi_n=r$, then $\pi_{n-1}=r+1.$
\end{itemize}
\end{lemma}
\begin{proof}
    For the first bullet point, note that $\pi_j=n$ implies that $\hat\pi_n=j$, which in turn implies that $\pi_{j+1}=j$ (since $nj$ must follow $n$ in the standard cycle form of $\hat\pi$). For the second bullet point, if $\pi_n=r$, then we must have that $\hat\pi_r=n$ and that $\hat\pi_{r+1}=r$, implying the result.  
\end{proof}

\begin{theorem}
    For $n\geq 5$, 
    \[
    f_n^2(\sigma) = \begin{cases}
        4 & \sigma=213,\\
        3 & \sigma=132,\\
        0 & \sigma=123.
    \end{cases}
    \]
\end{theorem}

\begin{proof}
First consider the case when $\sigma=213.$
Notice that there are exactly $f_{n-1}^2(213)$ permutations in $\mathcal{F}_n^2(213)$ that have 1 as a fixed point since adding or removing 1 as a fixed point will not change whether the permutations avoids 213 or whether a permutation is fixed under $\theta^2.$ Let us show that for $n\geq 5,$ all permutations in $\mathcal{F}_n^2(213)$ have 1 as a fixed point. Since $f_4^2(213)=4,$ the result would follow. 

For the sake of contradiction, suppose $\pi\in\mathcal{F}_n^2(213)$ with $\pi_1\neq 1$ and $n\geq 5$, and let $\hat\pi=\theta(\pi).$ 
First, note that since $\pi$ avoids 213 and $\pi_1\neq 1$, we must have that $n$ appears to the left of 1 in the one-line form of $\pi.$ 
Let us consider cases by $j$ where $\pi_j=n$. If $j=1,$ then by the first bullet point of Lemma \ref{lem: order2} we have $\pi=n1\pi_3\ldots\pi_n$ and thus $\hat\pi = \pi_3\ldots \hat\pi_{r-1}nr\ldots21$. Notice that $r>2$ since if $r=2$, we would have $n-1=r+1$ and thus $n=4$, which contradicts that $n\geq 5$. Suppose $\hat\pi_\ell=n-1$. Then by the second bullet point of Lemma \ref{lem: order2}, we have $\hat\pi_{\ell+1}=r+1.$ Therefore, $\pi$ contains the consecutive subsequence $\ell(n-1)2$ and ends with $(\ell+1)(r+1)r,$ which means that $\ell2(\ell+1)$ is a 213 pattern in $\pi,$ which is a contradiction.

Next, consider the case where $j\geq 2$. If $n-1$ does not appear before $n$ in $\pi$, then we have $\pi=\pi_1\ldots\pi_{j-1}nj\ldots (n-1)\ldots \pi_n$. Since $\pi$ avoids 213, it must be the case that $\pi_1<\pi_2<\ldots<\pi_{j-1}$ which implies that $\pi_{j-1}>j-1$, and thus $\pi_{j-1}>j$. But then $\pi_{j-1}j(n-1)$ is a 213 pattern. If $n-1$ does appear before $n$, it must be that $\pi_{j-1}=n-1$. But then $n-1$ is a fixed point of $\hat\pi$, so $\hat\pi$ ends with $(n-1) j$. Considering the cycle form of $\pi$, we therefore see that $\pi_{n-1}=j$, which by the first bullet point of Lemma \ref{lem: order2} implies that $j=n-2$. This combined with the second bullet point implies that $\pi$ ends with $n (n-2) (n-3)$. Since $1$ is not a fixed point of $\pi$, we must have that $n-3 = 1$. But this forces $n = 4$, so in particular there are no permutations in $\mathcal{F}_n^2(213)$ for $n\geq 5$ that do not have 1 as a fixed point.

Next, let us consider the case when $\sigma=132.$ The proof is similar. We show that when $n\geq 5,$ there are no permutations in $\mathcal{F}_n^2(132)$ that do not have $n$ as a fixed point. Since $f_4^2(132)=3$ and there are clearly $f_{n-1}^2(132)$ permutations in  $\mathcal{F}_n^2(132)$  that do have $n$ as a fixed point, the result would follow.

As before, for the sake of contradiction, we suppose $\pi\in\mathcal{F}_n^2(132)$ with $n\geq 5$ and $\pi_n\neq n$. First, suppose $\pi_1=n$. In that case, $\hat\pi_n=1$ and thus by Lemma \ref{lem: order2}, $\pi_2=1.$ Since $\pi$ avoids 132, $\pi = n123\ldots(n-1)$, which implies $\hat\pi=\theta(\pi) = n(n-1)\ldots 321,$ but $\hat\pi=\theta^{-1}(\pi) = 23\ldots n1$, which is a contradiction if $n>2.$  Now suppose $\pi_2=n$. Then $\pi_1=n-1$ and $\pi_3=2$. This implies that every element after 2 is increasing except for the element 1. So $\pi=(n-1)n23\ldots r1(r+1)\ldots (n-2)$ for some $3\leq r\leq n-2$ meaning that $\pi_{n-1}$ must be either 1 or $n-3$, both of which contradict Lemma \ref{lem: order2} since $n\geq 5$.

Next suppose $\pi_j=n$ for $2<j<n.$ Then $n-1$ appears before $n.$ Furthermore, $\pi_{j+1}=j$ by Lemma \ref{lem: order2}. This means that $\hat\pi_{n-1}\hat\pi_n=(j+1)j.$ Thus, in $\pi,$ $j+1$ follows $n-1$, and thus appears before $n.$ If $\pi_1\neq n-1,$  suppose $r$ is the position so that $\pi_r = j+1$. This means $\hat\pi_{n-2}=r$, and so $r$ follows $n-2$. Furthermore, since $n-1$ appears before $n$ and $j+1$ follows $n-1$, we have $r < j$.  However, since $\pi_1 \neq n-1$, we must have $n-2$ appear before $n-1$ to avoid $132$, but then $r(j+1)j$ is a 132 pattern in $\pi$. If $\pi_1=n-1$, then we must have that $\pi_2=j+1,$ which must be the smallest element that appears before $n$, meaning the elements $\pi_2\pi_3\ldots\pi_{j-1}$ are increasing, so $\pi_{j-1}=n-2.$ Thus $\hat\pi_{n-2}=n-1,$ but also $\hat\pi_{n-2}=2$ since $\pi_2=j+1$, giving a contradiction since $n\geq 5.$

Finally, let us consider the case when $\sigma = 123$. 
We first consider the case where $\pi_1 = n$. By Lemma \ref{lem: order2} this means $\pi_2 = 1$. This forces $\pi = n1(n-1)(n-2) \ldots 2$. Since $\theta(\hat\pi) = \pi$, this implies that $\hat\pi$ ends with $n21$, which in turn implies that $\theta(\hat\pi)$ ends with $n-2$. However, since $\pi=\theta(\hat\pi)$, this implies that $n-2=2$, so $n=4$, contradicting that $n\geq 5$. Furthermore, there is no $\pi$ with $\pi_n=n$ since such a permutation would be of the form $(n-1)(n-2)\ldots 21n,$ which does not satisfy that $\theta^2(\pi)=\pi$ when $n\geq4.$ We also cannot have $\pi_2=n$ since in this case $\hat\pi$ ends in 32, and so $3$ follows $n-1$ in $\pi$, giving us $\pi = \pi_1n2\ldots(n-1)3\ldots$. However if anything besides 1 appears between 2 and $n-1$ or after 3, we would have a 123 pattern and thus there are no such permutations for $n>6.$ So let us assume that $\pi_j=n$ for some $2<j<n.$ In the cases below, we show that there are no permutations in $\mathcal{F}^2_n(123)$ for $n>10$, but it is easily checked by computer that there are none for $5\leq n\leq 10$ as well. 


First, we consider the case when $n-1$ appears before $n$ in $\pi.$ In this case since $\pi$ avoids 123, we must have $\pi_1=n-1$. By Lemma~\ref{lem: order2}, we have \[\pi = (n-1)\pi_2\ldots\pi_{j-1}nj\pi_{j+2}\ldots\pi_{n-2}(r+1)r\] for some $r$ where $\pi_2>\pi_3>\ldots>\pi_{j-1}.$ Now, since $\hat\pi_{n-1}=j+1,$ we have that $\pi_2 = j+1,$ implying that $\hat\pi_{n-2}=2$, so $(n-2)2$ appears consecutively in $\pi.$ If it appears after $n$ in $\pi$, then $\pi_{j-1}j(n-2)$ is a 123 pattern as long as $j-1\geq 3$. If $j=3,$ then we have $\pi=(n-1)4n3\ldots(n-2)2\ldots.$ Since 1 is the only element that can appear between 3 and $(n-2)$, we must have $\pi_6=2$ or $\pi_7=2$ implying that $(n-3)6$ or $(n-3)7$ appears consecutively, meaning we cannot avoid 123 if $n>10.$ Finally, if $(n-2)2$ appears before $n$ in $\pi$, either we have $(j+1)(n-2)n$ as a 123 pattern or we have $j+1=n-2$, in which case $j=n-3$, so we have $(n-1)(n-2)2\ldots n(n-3)\pi_{n-1}\pi_n$ which cannot avoid 123 if $n>8.$

Now let us consider the case where $n-1$ appears after $n$ in position $\ell>j$. We first argue that $j < n-3$. If $j = n-1$, 
then $\pi$ ends in $n (n-1)$, which means $\pi = (n-2)(n-3) \cdots 21 n (n-1)$. However, we have $\theta^2(23 \cdots (n-2) 1 n (n-1))=(n-2)(n-3) \cdots 21 n (n-1)$, so $\pi$ is not fixed under $\theta^2$ when $n>4.$
If $j = n-2$, by Lemma \ref{lem: order2}, $\pi$ ends in $n (n-2) (n-1)$, but when $n \geq 4$ this clearly contains a $123$ pattern, namely $\pi_1(n-2)(n-1)$. If $j = n-3$, by both parts of Lemma~\ref{lem: order2}, $\pi$ must end with $n(n-3)(n-1)(n-2)$, and thus contains a $123$ when $n > 4$, namely $\pi_1(n-3)(n-1)$.

We now consider when $2 < j < n-3$, so that we have \[
\pi = \pi_1\ldots \pi_{j-1}nj\pi_{j+2}\ldots \pi_{\ell-1}(n-1)(j+1)\pi_{\ell+2}\ldots \pi_n.
\]  

So $\hat\pi$ ends with $(\ell+1)(j+1)j$, meaning that $(n-2)(\ell+1)$ must appear consecutively in $\pi$ unless $\ell+1=n-2,$ in which case $n-2$ is a fixed point of $\hat\pi$ and so must appear in position $j-1$ in $\pi$. Let's consider this latter case first. In this case, since $\pi_{j-1}\pi_j=(n-2)n$, then since $j>2$, $\pi_1(n-2)n$ is a 123 pattern. For the former case, we must not have $(n-2)(\ell+1)$ appear after $n-1$ since $j(j+1)(\ell+1)$ would be a 123 pattern. It cannot appear after $n$ and before $n-1$ since we would then have the 123 pattern given by $j(\ell+1)(n-1)$ as we know $j \not= n-2$. Thus it appears before $n$, and so $\pi_1\pi_2=(n-2)(\ell+1)$, and so $\hat\pi_{n-3}=2,$ meaning $(n-3)2$ appears consecutively in $\pi$. But this implies that either $2j(j+1)$, $j(j+1)(n-3)$, or $j(n-3)(n-1)$ is a 123 pattern in $\pi.$
\end{proof}


We end this section with a few conjectures concerning pattern-avoiding permutations fixed under $\theta^k$ for $k>2.$

\begin{conjecture}
    For $\sigma\in\{231,312\}$, the generating functions $F_\sigma^k(x)$ are rational. A few examples of conjectured generating functions for these patterns are found below.
    \[F_{\sigma}^3(x) = \frac{1}{1 - x - x^2 - 2x^3},\]
\[F_{\sigma}^4(x) = \frac{1}{1 - x - x^2 - 2x^4-x^5-x^6},\]
\[F_{\sigma}^5(x) = \frac{1}{1 - x - x^2},\]
\end{conjecture}

\begin{conjecture}
    For $\sigma\in\{213,132\}$ and $k \geq 1$, $f_n^k(\sigma)$ is eventually constant for large enough values of $n$. Below is a table of the conjectured values when $n\geq 8$ and $1\leq k \leq 14.$.
    \begin{center}
    \begin{tabular}{c|c|c|c|c|c|c|c|c|c|c|c|c|c|c}
         $k$&  $1$& $2$& $3$ & $4$& $5$& $6$& $7$ & $8$ &$9$ & $10$ & $11$ & $12$ & $13$ &$14$\\
         \hline\hline
         $f_n^k(213)$ &$2$ &$4$ &$7$ &$9$ &$2$ &$9$ &$8$ &$9$ &$8$ &$7$ &$2$ &$16$ &$2$ &$10$\\
         \hline
         $f_n^k(132)$ & $2$ & $3$ & $6$ & $7$ & $2$ & $7$ & $8$ & $7$ & $7$ & $5$ & $2$ & $13$ & $2$ & $9$
    \end{tabular}
    \end{center}
\end{conjecture}

Oddly, we notice that some periodicity seems to appear as we increase $k$. For example, for any fixed $n \geq 1$ and $\sigma \in S_3$, $f_n^k(\sigma)$ appears to be the same for \[k \in \{1,5,11,13, 19, 23, 29, 31, 37, 41, 43, 53, 55, 59, 65, 67, 71, 73, 79, 83, 89, 95, 97,\dots\},\]
as well as for $k\in\{2, 26, 46, 58, 62, 74, 82, 86,\dots\}$, for $k\in\{3, 15, 33, 39, 57, 69, 87, 93,\dots\},$ and for $k\in\{4, 8, 16, 32, 52, 64, 92,\dots\}$. 
This leads to the general question below.

\begin{question}
    For a given $\sigma \in S_3$, when is $f_n^i(\sigma) = f_n^j(\sigma)$ for all $n\geq 1$?
\end{question}

There are certain patterns one might recognize in these values of $k$. One such possibility are the conjectures below.

\begin{conjecture}
    For $i,j \geq 2$, $n \geq 1$ and $\sigma \in S_3$, $f_{n}^{2^i}(\sigma) = f_n^{2^j}(\sigma)$.
\end{conjecture}
\begin{conjecture}
    For $i\geq 1$, $n \geq 1$ and $\sigma \in S_3$, $f_{n}^{i}(\sigma) = f_n^{1}(\sigma)$ if and only if $f_{n}^{3i}(\sigma) = f_n^{3}(\sigma)$ .
\end{conjecture}

\section{Further research} \label{sec: Conclusion}

There are some natural directions for future research in this area, including enumerating permutations that avoid a pattern or set of patterns whose image under $\theta$ avoids another pattern or set of patterns. One might also consider consecutive and vincular patterns (especially in light of the characterization of shallow permutations in terms of vincular patterns, as mentioned in the introduction).

One might also consider questions about the fundamental bijection unrelated to pattern avoidance. Let $\mathcal{F}_n^k$ denote the number of permutations in $\mathcal{S}_n$ fixed by $k$ under the group action $\star$ described in Section~\ref{sec: fixedOrder} (i.e., those permutations with $\pi=\theta^k(\pi)$), and let $f_n^k = |\mathcal{F}_n^k|$. As we mentioned in the previous section, the results on order two permutations required a lot of case work despite their simple answers. We note that another approach to this problem is to first understand the structure of permutations $\pi \in \mathcal{S}_n$ such that $\pi = \theta^2(\pi)$, an analogue of Lemma \ref{lem:fixed}. Equipped with this tool, one should be able to more easily determine which of these permutations avoid a given pattern. This is also an interesting question in its own right.

\begin{question}
    How many $\pi \in \mathcal{S}_n$ satisfy $\pi = \theta^2(\pi)$? 
\end{question}

In other words, what is $f_n^2$? Is there a nice way to understand how the permutations in $\mathcal{F}_n^2$ decompose into cycles, as in Lemma \ref{lem:fixed} for those permutations fixed under $\theta$? We are able to prove a lower bound:

\begin{proposition}
    For $n > 9$,
    \[
    f_n^2 \geq f_{n-1}^2 + f_{n-2}^2 + 2 f_{n-4}^2 + 2 f_{n-9}^2.
    \]
\end{proposition}

\begin{proof}
    By Lemma \ref{lem: thetaCommutes} if $\pi = \bigoplus_{i=1}^m \tau_i$ where the $\tau_j$ are irreducible, then $\theta^2(\pi) = \pi$ implies that $\theta^2(\tau_j) = \tau_j$ for each $j$. Furthermore, given any irreducible $\tau_1,\dots,\tau_m$ with $\theta^2(\tau_j) = \tau_j$, $\pi = \bigoplus_{i=1}^m \tau_i$ satisfies $\theta^2(\pi) = \pi$. So it suffices to determine the number of irreducible $\tau$ with $\tau = \theta^2(\tau)$. Any $\pi \in \mathcal{F}_n^2$ will be a direct sum of these $\tau$.

    To prove the stated lower bound, it suffices to find the irreducible permutations that $\pi$ could end with. The result then follows from noticing that $\tau = 1$, $\tau = 21$, $\tau = 3421$, $\tau = 4132$, $\tau = 637948521$ and $\tau = 916823754$ are all irreducible permutations with $\theta^2(\tau) = \tau$. So there is one such permutation of sizes $1$ and $2$, and two such permutations of sizes $4$ and $9$. This establishes the desired lower bound.
\end{proof}

It is possible this inequality is sharp. This would true if one could prove that for $n \geq 10$ there are no irreducible permutations $\tau\in\S_n$ with $\theta^2(\tau) = \tau$. We note that if and irreducible permutation $\tau$ satisfies $\theta^2(\tau) = \tau$ then so does its (irreducible) image $\theta(\tau)$. Since for irreducible permutations, $\theta(\tau) = \tau$ only for $\tau$ or size $1$ or $2$ by Lemma \ref{lem:fixed}, all other such irreducible $\tau$ will come in distinct pairs. 
We present the following table of $f_n^2$ values for $n \leq 11$:

\begin{center}
    \begin{tabular}{c|c|c|c|c|c|c|c|c|c|c|c}
         $n$&  $1$& $2$& $3$ & $4$& $5$& $6$& $7$ & $8$ &$9$ & $10$ & $11$ \\
         \hline
         $f_n^2$ &$1$ &$2$ &$3$ &$7$ &$12$ &$23$ &$41$ &$78$ &$145$ &$271$ &$502$\\
    \end{tabular}
\end{center}

Naturally, you can extend this open question to any number of iterations of $\theta,$ asking: How many permutations $\pi\in\S_n$ satisfy $\pi=\theta^k(\pi)$?

\end{document}